\documentclass[amssymb, 11pt]{amsart}
\usepackage{latexsym}
\usepackage{young}
\usepackage{hyperref}
\newcommand{\qm}[1]{\quad\mbox{#1}}

\newlength{\standardunitlength}
\newcommand{\bea}{\begin{eqnarray}}
\newcommand{\ena}{\end{eqnarray}}
\newcommand{\beas}{\begin{eqnarray*}}
\newcommand{\enas}{\end{eqnarray*}}
\newcommand{\qmq}[1]{\quad \mbox{#1} \quad}
\newcommand{\Bvert}{\left\vert\vphantom{\frac{1}{1}}\right.}

\newcommand{\ignore}[1]{}

\newtheorem{prop}{Proposition}[section]

\newtheorem{theorem}[prop]{Theorem}

\begin{document}

\title [Stein's method and the semicircle distribution] {Stein's method, semicircle distribution, and reduced decompositions of the longest element in the symmetric group}

\author{Jason Fulman}

\address{Department of Mathematics\\
        University of Southern California\\
        Los Angeles, CA, 90089}
\email{fulman@usc.edu}

\author{Larry Goldstein}
\address{Department of Mathematics\\University of Southern California\\
        Los Angeles, CA, 90089}
\email{larry@math.usc.edu}

\thanks{Fulman was supported by NSA grant H98230-13-1-0219.}


\date{May 3, 2014}

\begin{abstract} Consider a uniformly chosen random reduced decomposition of the longest element in the symmetric group. It is known
that the location of the first transposition in this decomposition converges to the semicircle distribution. In this note we provide a sharp error term for this result, using the ``comparison of generators'' approach to Stein's method.
\end{abstract}

\maketitle

\section{Introduction} \label{into}

The length of a permutation is the minimal number of terms needed to express the permutation as a product of adjacent transpositions. The longest element in the symmetric group $S_n$ is the permutation \[ n \ n-1  \cdots \ 2 \ 1 , \] and it has length ${n \choose 2}$. For $1 \leq s \leq n-1$, denote the adjacent transposition at location $s$ by $\tau_s = (s \ s+1)$. A {\it reduced decomposition} of the longest element is defined as a sequence $s_1,\cdots,s_{{n \choose 2}}$ such that the longest element is equal to $\tau_{s_1} \tau_{s_2} \cdots \tau_{s_{{n \choose 2}}}$.

There are nice enumerative results about reduced decompositions of the longest element. Stanley \cite{St} proved that the number of reduced decompositions of the longest element of $S_n$ is equal to \[ \frac{{n \choose 2}!}{1^{n-1} 3^{n-2} 5^{n-3} \cdots (2n-3)^1} .\]  This formula was later proved bijectively by Edelman and Greene \cite{EG}. A gentle introduction to the enumerative theory of reduced decompositions of the longest element (and of other permutations) can be found in Chapter 7 of \cite{BS}. Garsia's lecture notes \cite{Ga} are also very informative.

Given these enumerative results, it is very natural to study uniformly chosen random reduced decompositions of the longest element. For example, a reduced decomposition of the longest element is said to have a ``Yang-Baxter'' move at position $k$ if $s_k,s_{k+1},s_{k+2}$ is equal to $j,j+1,j$ or to $j+1,j,j+1$ for some $j=1,2, \cdots n-2$. Reiner \cite{R} proves that for all $n \geq 3$, the expected number of Yang-Baxter moves of a random reduced decomposition of the longest element of $S_n$ is equal to 1. He conjectures that for $n \geq 4$, the variance of the number of Yang-Baxter moves is $\frac{{n \choose 2}-4}{{n \choose 2} -2}$, and that as $n \rightarrow \infty$, the distribution of the number of Yang-Baxter moves approaches that of a Poisson random variable with mean $1$. Reiner's conjecture seems challenging, but it is conceivable that Stein's method for Poisson approximation (\cite{AGG},\cite{CDM}) will be helpful. The results of the current paper show that the ``comparison of generators'' approach to Stein's method is useful for studying random reduced decompositions of the longest element.

Reduced decompositions of the longest element are also called {\it sorting networks} by Angel, Holroyd, Romik, and Virag \cite{AHRV}, who study them extensively. One of their results concerns the random location $X=s_1$ of the first adjacent transposition of the reduced decomposition of the longest element. Using the bijection of Edelman and Greene \cite{EG}, they show that for $1 \leq k \leq n-1$, the probability $p(k)=P(X=k)$ is given by
\begin{equation} \label{fi}
p(k)=\frac{1}{{n \choose 2}}\frac{(3 \cdot 5 \cdots (2k-1))}{(2 \cdot 4 \cdots (2k-2))} \frac{(3 \cdot 5 \cdots (2(n-k)-1))}{(2 \cdot 4 \cdots (2(n-k)-2))}.\end{equation} We say that a random variable $S$ has the semi-circle distribution (also called the Sato-Tate distribution by number theorists) if it has density function
\beas \frac{2}{\pi} \sqrt{1-s^2} \qmq{for $s \in (-1,1)$.} \enas
Using \eqref{fi} and Stirling's formula, the paper \cite{AHRV} shows that
\bea \label{limit.semi}
\frac{2X}{n} - 1 \rightarrow_d S \qmq{as $n \rightarrow \infty$,}
\ena where $\rightarrow_d$ denotes convergence in distribution. The main result of this note provides upper and lower bounds of the same order on the Wasserstein distance between $2X/n-1$ and its semi-circle distributional limit.

We recall that the Wasserstein distance between two random variables $U$ and $V$ is given by
\bea \label{w.as.sup}
d_W(U,V) = \sup_{h\in {\mathcal L}_1 }|Eh(U)-Eh(V)|,
\ena
where ${\mathcal L}_1$ is the collection of all $1-$Lipschitz functions $h$, that is, those functions $h$ that satisfy $|h(x)-h(y)| \le |x-y|$ for all real $x,y$. Alternatively, see \cite{Ra} for instance, we have
\bea \label{alter.wass}
d_W(U,V) = \inf_P E|U-V|
\ena
where the infimum is taken over all joint distributions $P$ that have the given marginals.

\begin{theorem} \label{mainthm}
For $X$ with distribution \eqref{fi} and $S$ having the semi-circle distribution, \beas
\frac{1}{16n} \le d_W\left(\frac{2X}{n}-1,S\right) \le \frac{59}{n}.
\enas
\end{theorem}

The reader may wonder why there is a need for Theorem \ref{mainthm}, given that there is an explicit formula \eqref{fi}. The answer is that explicit formulas are not always so informative; just as it is interesting to quantify the error in the normal approximation to the (explicit) binomial distribution, it is interesting to quantify the error of approximation by \eqref{limit.semi}.

We mention that there is some prior work relating Stein's method to the semicircle law. Namely G\"otze and Tikhomirov \cite{GT} show that the expected spectral distribution of certain random matrices converges to the semicircle law by using Stein's method to study the Stieltjes transform of the empirical spectral distribution function of random matrices. The methods we employ are completely different; we use a relationship between the semicircle law and the Beta distribution, and the ``comparison of generators'' approach to the Beta distribution from \cite{GR}.

The work in \cite{GR} applied here uses a `discrete density' approach to Stein's method to obtain Proposition \ref{c.at.a.is.zero} below, and as noted in \cite{GR}, shares similarity to the approach taken in \cite{LS}.

\section{Main results} \label{main}
We apply the ideas in \cite{GR}, where the Beta approximation to the P\'olya urn distribution was studied. The support $\{0,\ldots,n\}$ of the examples studied in \cite{GR} corresponds nicely after scaling to the support of the limiting Beta distribution, whereas here the random variable $X$ has support $\{1,\ldots,n-1\}$. To handle the current situation, we provide the following slight generalization of one direction of Corollary 2.1 of \cite{GR} which removes a boundary condition on functions $f$ for the satisfaction of \eqref{c.identity}. In the following, for any $q \in\mathbb{R}$ let $\Delta_q f(k)=f(k+q)-f(k)$ and let $\Delta=\Delta_1$ the forward difference.

\begin{prop} \label{c.at.a.is.zero}
Let $p$ be a probability mass function with support the integer interval $I=[a,b] \cap \mathbb{Z}$ where $a \le b, a,b \in \mathbb{Z}$, and let $\psi$ be given by
\bea \label{def:psi}
\psi(k)= \frac{\Delta p(k)}{p(k)}, \quad k \in I.
\ena
Then for all functions $c: [a-1,b] \cap \mathbb{Z} \rightarrow \mathbb{R}$ satisfying $c(a-1)=0$, if $Y$ is a random variable with probability mass function $p$ then
\bea  \label{c.identity}
E[c(Y-1)\Delta f(Y-1)+[c(Y)\psi(Y)+c(Y)-c(Y-1)]f(Y)]=0
\ena
for all functions $f:[a-1,b] \rightarrow \mathbb{R}$.
\end{prop}

\begin{proof} By (11) in the proof of Proposition 2.1 of \cite{GR}, we see that
\bea \label{characterization.one}
E\left[ \Delta f(X-1) + \psi(X) f(X) +f(a-1){\bf 1}(X=a) \right] =0
\ena
for all functions $f:[a-1,b] \rightarrow \mathbb{R}$. Replacing $f(k)$ by $c(k)f(k)$ and using $c(a-1)=0$ yields the result.
\end{proof}

We next note that may write \eqref{fi} a bit more compactly and express the distribution of $X$ as
\bea \label{def.pk}
p(k) = \frac{1}{{n \choose 2}} \frac{\prod_{j=1}^{k-1}(2j+1)}{\prod_{j=1}^{k-1}2j} \frac{\prod_{j=1}^{n-k-1}(2j+1)}{\prod_{j=1}^{n-k-1}2j}, \qmq{$1 \le k \le n-1$.}
\ena

We will find it more convenient to make a linear transformation on both sides of \eqref{limit.semi} and deal instead with a limiting Beta distribution; we recall that the Beta distribution ${\mathcal B}(\alpha,\beta)$ with positive parameters $\alpha,\beta$ is supported on $(0,1)$ with density there proportional to $x^{\alpha-1} (1-x)^{\beta-1}$.

By \eqref{alter.wass} one may directly verify that the Wasserstein distance satisfies the scaling property
\bea \label{W.scales}
d_W(aU+b,aV+b)=|a|d_W(U,V) \qmq{for all real $a,b$.}
\ena
Next, it is easy to verify that
\bea \label{2Z-1}
2Z-1=_d S
\ena
when $Z \sim {\mathcal B}(3/2,3/2)$ and $S$ has the semicircle distribution, where $=_d$ denotes equality of distribution.
\ignore{
The density $f_Z(z)$ of the linear transformation of $S$ is given by
\begin{multline*}
f_Z(z)=\frac{dP(Z \le z)}{dz} = \frac{dP((S+1)/2 \le z)}{dz}
= \frac{dP(S\le 2z-1)}{dz} \\ = 2f_S(2z-1) = \frac{4}{\pi}\sqrt{1-(2z-1)^2} = \frac{8}{\pi}z^{1/2}(1-z^2)^{1/2}=\frac{z^{1/2}(1-z^2)^{1/2}}{{\cal B}(\frac{3}{2},\frac{3}{2})}.
\end{multline*}
}
From \eqref{W.scales} and \eqref{2Z-1} we see that Theorems \ref{mainthm} and \ref{thm:rsn} are equivalent.

\begin{theorem} \label{thm:rsn}
Let $W_n=X/n$ with $X$ having distribution \eqref{def.pk} and let $Z \sim {\mathcal B}(\frac{3}{2},\frac{3}{2})$. Then
\beas
\frac{1}{32n} \le d_W(W_n,Z) \le \frac{59}{2n}.
\enas
\end{theorem}

\begin{proof} We apply Proposition \ref{c.at.a.is.zero} to the distribution of $X$, supported on $[1,n-1] \cap \mathbb{Z}$. To calculate $\psi$ in \eqref{def:psi}, by \eqref{def.pk}, for $1 \le k \le n-1$ we have
\begin{multline*}
p(k+1)
= \left( \frac{2k+1}{2k} \right) \left( \frac{2(n-k)-2}{2(n-k)-1} \right) p(k) = \frac{(2k+1)(n-k-1)}{k(2(n-k)-1)} p(k),
\end{multline*}
so that
\begin{eqnarray*}
\psi(k)  = \frac{(2k+1)(n-k-1)}{k(2(n-k)-1)} - 1 = \frac{n-2k-1}{k(2(n-k)-1)}, \qm{$1 \le k \le n-1$.}
\end{eqnarray*}
Let
\beas
c(k)=k(2(n-k)-1).
\enas
Then
\beas
c(k-1) = (k-1)(2(n-k+1)-1) = (k-1)(2(n-k)+1)
\enas
and
\beas
c(k)\psi(k)+c(k)-c(k-1)
=-6k+3n=3(n-k)-3k.
\enas
Hence, noting $c(0)=0$, Proposition \ref{c.at.a.is.zero} yields
\beas
E\left[ (X-1)(2(n-X)+1) \Delta f(X-1) +[3(n-X)-3X] f(X) \right]=0
\enas
for all $f:[0,n-1] \cap \mathbb{Z} \rightarrow \mathbb{R}$. Replacing $f(k)$ by $f(k/n)$, dividing by $2n$ and now writing $W=X/n$ we obtain
\begin{multline}
E \left[ \left(nW-1\right) \left(1-W+\frac{1}{2n} \right) \Delta_{1/n} f(W - \frac{1}{n}) + \left(\frac{3}{2}(1-W)-\frac{3}{2}W\right) f(W) \right]\\
=0 \label{rsn.identity}
\end{multline}
for all $f: \{k/n: 0 \le k \le n-1, k \in \mathbb{Z}\} \rightarrow \mathbb{R}$. Applying \eqref{rsn.identity} by letting $f(x)$ be the indicator of the set $\{k/n: k \ge 1\}$ which contains the support of $W$, noting that $\Delta_{1/n}f(w-1/n)={\bf 1}(w=1/n)$ we obtain
\bea \label{EW}
E \left[ \frac{3}{2}(1-W)-\frac{3}{2}W  \right] =0 \qmq{and thus} EW=1/2.
\ena

Given an absolutely continuous test function $h$ on $[0,1]$, letting ${\mathcal B}h$ be the ${\mathcal B}(\frac{3}{2},\frac{3}{2})$ expectation of $h$, Lemma 3.1 of \cite{GR} shows that there exists a unique bounded solution $f$ on $[0,1]$ that solves the Beta Stein equation,
\beas
 w(1-w)f'(w)+(\alpha(1-w)-\beta w)f(w)=h(w)-{\mathcal B}h.
\enas
By Lemmas 3.2 and 3.4 in \cite{GR}, now letting $f$ take the value $0$ outside the unit interval,
$f$ obeys the bounds
\bea \label{bounds.on.f}
||f|| \le \frac{2}{3}||h'|| \qmq{and} ||f'|| \le 8||h'||,
\ena
where $||\cdot||$ denotes the supremum norm.

Using that $f$ is the solution to the Stein equation for $h$ for the first equality, and identity \eqref{rsn.identity} for the second, we have
\begin{multline}
Eh(W)- {\mathcal B} h = E\left[ W(1-W)f'(W)+\left(\frac{3}{2}(1-W)-\frac{3}{2} W\right)f(W) \right] \\
= E\left[ W(1-W)f'(W)-\left(nW-1\right) \left((1-W)+\frac{1}{2n} \right) \Delta_{1/n} f(W-1/n) \right]\\
= E\left[ W(1-W)f'(W)- nW (1-W) \Delta_{1/n} f(W-1/n)\right]  +R_1 \label{R1}
\end{multline}
where
\beas
R_1 = E\left[ \left( 1 + \frac{1}{2n} - \frac{3}{2}W \right)\Delta_{1/n}f(W-1/n)\right] .
\enas
As the mean value theorem yields $|\Delta_q f(x)| \le q||f'||$ for all $q>0$ and $x \in \mathbb{R}$,  we obtain
\bea \label{R1.bound}
R_1 \le
\frac{1}{n}\left(1 + \frac{1}{2n} + \frac{3}{4} \right) ||f'|| \le \frac{9}{4n}||f'||,
\ena
where we have applied \eqref{EW}.

Now returning to the first term in \eqref{R1}, collecting terms (3.6), (3.7), (3.8) and (3.9) from the proof of Theorem 1.1 in \cite{GR} specialized to $\alpha=\beta=3/2$ and $m=1$, we have
\begin{multline}
\Bvert E\left[ W(1-W)f'(W)-nW(1-W)\Delta_{1/n}f(W-1/n) \right] \Bvert \\
\le
\frac{1}{2n}||f'||+\frac{1}{2n}||h'||+\frac{3}{4n}||f'||+\frac{3}{2n}||f||. \label{int.bound}
\end{multline}

Combining bounds \eqref{R1.bound} and \eqref{int.bound} and then applying \eqref{bounds.on.f} we obtain
\beas
\vert Eh(W)-{\mathcal B}h \vert \le \frac{1}{n}\left( \frac{14}{4}||f'||+ \frac{1}{2}||h'||+\frac{3}{2} ||f||\right) \le \frac{59}{2n}||h'||.
\enas
Taking supremum over functions $h$ satisfying $||h'|| \le 1$ yields the upper bound in Theorem \ref{thm:rsn}.

For the lower bound, applying \eqref{rsn.identity}
for $f(x)=x$ one has $\Delta_{1/n}f(x)=1/n$, and
\beas
E \left[ \left(W-1/n\right) \left(1-W+\frac{1}{2n} \right) + \left(\frac{3}{2}(1-W)-\frac{3}{2}W\right) W \right] =0.
\enas
Solving for $EW^2$ and using $EW=1/2$ from \eqref{EW}, we obtain
\beas
E\left( \frac{1}{2}W^2 \right) =\frac{5}{32}-\frac{2+n}{32n^2}.
\enas
It is easy to verify that
\beas
E\left( \frac{1}{2}Z^2 \right) = \frac{5}{32}.
\enas
As the function $h(w)=w^2/2$ is an element of ${\mathcal L}_1$ on $[0,1]$, we obtain the claimed lower bound from \eqref{w.as.sup}, as
\beas
d_W(W,Z) \ge \Bvert E\left( \frac{1}{2}W^2 \right) - E\left( \frac{1}{2}Z^2 \right) \Bvert = \frac{2+n}{32n^2} \ge \frac{1}{32 n}.
\enas
\end{proof}

\end{document}